%% file: main.tex
\documentclass[]{interact}

\usepackage{epstopdf}
\usepackage[caption=false]{subfig}
\usepackage{xcolor}
\usepackage{enumitem}
\usepackage{diagbox}

\usepackage[numbers,sort&compress]{natbib}
\bibpunct[, ]{[}{]}{,}{n}{,}{,}
\makeatletter
\def\NAT@def@citea{\def\@citea{\NAT@separator}}
\makeatother

\theoremstyle{plain}
\newtheorem{theorem}{Theorem}[section]

\newtheorem{corollary}[theorem]{Corollary}
\newtheorem{proposition}[theorem]{Proposition}

\theoremstyle{definition}
\newtheorem{definition}[theorem]{Definition}

\theoremstyle{remark}
\newtheorem{remark}[theorem]{Remark}

\input{macros}

\begin{document}

\title{Invertibility of circulant matrices of arbitrary size}

\author{
\name{Jeong-Ok Choi\textsuperscript{a,} \textsuperscript{b} and Youngmi Hur\textsuperscript{c}}
\affil{\textsuperscript{a}Division of Liberal Arts and Sciences, Gwangju Institute of Science and Technology, Gwangju, South Korea; \textsuperscript{b}Discrete Mathematics Group, Institute for Basic Science (IBS), Daejeon, South Korea; \textsuperscript{c}Department of Mathematics, Yonsei University, Seoul, South Korea}
}

\maketitle

\begin{abstract}
In this paper, we present sufficient conditions to guarantee the invertibility of rational circulant matrices with any given size. These sufficient conditions consist of linear combinations of the entries in the first row with integer coefficients. Our result is general enough to show the invertibility of circulant matrices with any size and arrangement of entries. For example, using these conditions, we show the invertibility of the family of circulant matrices with particular forms of integers generated by a primitive element in $\mathbb{Z}_p$. Also, using a combinatorial structure of these sufficient conditions, we show invertibility for circulant $0, 1$-matrices.
\end{abstract}


\begin{amscode}
15B05, 15B36, 11A07
\end{amscode}


\section{Introduction}

In this paper, we study the invertibility of circulant matrices of arbitrary size. Our research is motivated by a specific type of matrices that the second author encountered in an earlier research \cite{GMD,ZachThesis}. We denote these matrices by $A_{p,m}$ in this paper. The matrices $A_{p,m}$ are permutation similar to some circulant matrices which we call $G_{p,m}$ (c.f. Section~\ref{subS:GMD}). The determinant of $A_{p,m}$ is a generalization of the Maillet determinant \cite{TM,HWT,CO,CO2}, and plays an essential role in determining a certain type of approximation power for the multivariate wavelets constructed by a method called {\it prime coset sum} \cite{PCS}. Although further developments on the approximation power of the prime coset sum wavelets are made, the question about the invertibility of matrices $A_{p,m}$ is mostly left unanswered.

For a positive integer $m$ and an odd prime $p$, $A_{p,m}$ is the $(p-1)$ by $(p-1)$ matrix of the form
\begin{equation}
\label{eq:Apm}
A_{p,m}:=\left[((i^{-1}(\mod\; p)\cdot j)\,(\mod\; p))^{m}\right]_{i,j=1}^{p-1}
\end{equation}
where $i^{-1}\;(\mod\; p)$ is the multiplicative inverse of $i$ in the field $\mathbb{Z}_{p}$, and the exponentiation outside of the parentheses is taken as in the integers. 

The invertibility of circulant matrices under various constraints has been studied in the literature. 
The exact form of eigenvalues of circulant matrices is known (c.f. Theorem~\ref{thm:ams}, \cite{kra12,HJ}) and is used to provide a sufficient condition for a circulant matrix to be invertible, but many of these existing conditions cannot be directly applied to determine the invertibility of the circulant matrix $G_{p,m}$ (hence the invertibility of the above matrix $A_{p,m}$). For example, a circulant matrix needs to be of size $n$ by $n$ with prime $n$ for a result in \cite{kra12} (c.f. Corollary~\ref{cor:q}), and its entries need to be ``unbalanced'' for a result in \cite{HJ} (c.f. discussion after Theorem~\ref{thm:ams}). However, the size of $A_{p,m}$ is $(p-1)$ by $(p-1)$ with odd prime $p$, and its size is never a prime except for the trivial case $p=3$. The number $p-1$ with prime $p$ does not have any specific kind of factorization. Hence we need a method to tell us about the invertibility for a circulant matrix of arbitrary size. 

The main contribution of our paper is to provide a method for determining the invertibility for such an arbitrary size circulant matrix when its entries are rational numbers. Our method determines the invertibility, not just for specific types of circulant matrices such as $G_{p,m}$, but circulant matrices without particular patterns. Another family of circulant matrices studied in the literature is circulant $0, 1$-matrices that have useful applications and have a combinatorially interesting structure on their own. In \cite{so06}, the author characterized graphs with circulant adjacency matrices with integral spectra using the eigenvalue formula mentioned above and cyclotomic polynomials, but this characterization cannot tell whether those adjacency matrices are invertible or not. In Section~\ref{subS:01circulant}, we introduce a known result on invertibility of circulant $0, 1$-matrices with specific constraints (\cite{chen21}) and provide a related new result. 
\medskip

The rest of this paper is organized as follows. In Section~\ref{sec:Prelim}, we review elementary topics necessary for our main results. Section~\ref{sec:Main} shows our main theorem providing a sufficient condition for being invertible circulant matrices of arbitrary sizes with rational entries. Section~\ref{sec:Explicit} presents explicit conditions for various sizes, readily applied to any circulant matrices with rational entries. Using the conditions obtained in Section~\ref{sec:Explicit}, we apply them for different circulant matrices to show their invertibility in Section~\ref{sec:Appl}.


\section{Preliminaries}
\label{sec:Prelim}
In this section, we review terminologies and some results in abstract algebra and elementary number theory used for our main results on the invertibility of circulant matrices. The topics in this section are not necessarily related to one another, but we will use them together for our main result.
\medskip

\begin{definition} For a natural number $n$, the value of $\phi(n)$ is the number of positive integers $d$ up to $n$ that are relatively prime to $n$, i.e. $(d,n)=1$. The function $\phi: \mathbb{N} \longrightarrow \mathbb{N}$ is called the {\it Euler's totient function}. \end{definition}

\begin{definition} Let $w_1, w_2, \cdots, w_{\phi(n)}$ be the primitive $n$-th roots of unity. The polynomial $(x - w_1)(x-w_2) \cdots (x - w_{\phi(n)})$ is called the $n$-th {\it cyclotomic polynomial} over $\mathbb{Q}$ and denoted by $\Phi_n(x)$. \end{definition}

We refer \cite{lang} and \cite{gallian} for basic and well-known results on cyclotomic polynomials and extension fields, which can be found in algebra textbooks. 

 \begin{theorem}[\cite{lang}]
\label{thm:lang} For every positive integer $n$, $x^n - 1  = \prod_{d \mid n} \Phi_d(x)$, where the product is taken over all positive divisors $d$ of $n$.
\end{theorem}

For a field $F$ and an element $\alpha$ in an extension field $E$ of $F$, we denote by $F(\alpha)$ the smallest subfield of $E$ containing $F$ and $\alpha$. 

\begin{theorem}[\cite{gallian}]
\label{thm:gallian}
Let $F$ and $F'$ be fields, let $p(x) \in F[x]$ be irreducible over $F$, and let $\alpha$ be a zero of $p(x)$ in some extension of $F$. If $\varphi$ is a field isomorphism from $F$ to $F'$ and $\beta$ is a zero of $\varphi(p(x))$ in some extension of $F'$, then there is an isomorphism from $F(\alpha)$ to $F'(\beta)$ that agrees with $\varphi$ on $F$ and carries $\alpha$ to $\beta$. 
\end{theorem}

Let $1, w, w^2, \cdots, w^{n-1}$ be all the $n$-th root of unity. In other words, $w$ is a primitive $n$-th root of unity. For each integer $n > 1$, we consider $U(n) = \{ 1 \le d \le n \colon (d, n) = 1 \}$. Then its size is $\phi(n)$ and it is a group under multiplication modulo $n$. 

\medskip

For each $k \in U(n)$, $w^k$ is a primitive $n$-th root of unity and it is a root of the cyclotomic polynomial $\Phi_n(x)$. It is known that every cyclotomic polynomial is irreducible over $\mathbb{Q}$. So by letting $F = F' = \mathbb{Q}$ in Theorem~\ref{thm:gallian}, there exists a field automorphism of $\mathbb{Q}(w) = \mathbb{Q}(w^k)$ that maps $w$ to $w^k$, and its restriction on $\mathbb{Q}$ is the identity. Hence we obtain the following result. 

\begin{proposition}
\label{prop:cy}
 In any cyclotomic polynomial $\Phi_t(x)$, if $\epsilon_1, \epsilon_2, \cdots, \epsilon_{\phi(t)}$ are the zeros, then for each $j$ there exists an automorphism $f_j$ of $\mathbb{Q}(\epsilon_1)$ such that $f_j(\epsilon_1) = \epsilon_j$ fixing $\mathbb{Q}$. 
\end{proposition} 

Note that, without loss of generality, we may take the $d$-th primitive root of unity to be $e^{{2\pi i}/{d}}$.

\begin{definition}
For any positive integers $d$ and $n$, the sum of the $n$-th power of $d$-th primitive root of unity is called the {\it Ramanujan's sum} and denoted by $C_d(n)$. In other words,
\[C_d(n) =  \sum_{\genfrac{}{}{0pt}{}{1\le a \le d}{(a, d) = 1}}e^{2\pi i \frac{a}{d}n} \ \ . \]
\end{definition}

It is trivial to see that $C_d(n)$ is periodic modulo $d$: $C_d(n) = C_d(n+d)$ for every $d, n \ge 1$. Clearly, $C_1(n)=1$ for every $n \in \mathbb{N}$. There is no known explicit formula for $C_d(n)$ for arbitrary natural numbers $d$ and $n$ in general, but some formulas are known. We list relevant properties that we will use later. 

\medskip


{\bf Some known properties about Ramanujan's sum} (See \cite{nicol62, lucht10, anderson53, hardy27}.)
\begin{enumerate}[label=(\arabic*)]
\item 
$C_d(n)$ is an integer for all $d, n \in \mathbb{N}$. 
\item
$C_d(d) = \phi(d)$ for every $d \in \mathbb{N}$.
\item \label{Ram:prod}
If $(d, d') = 1$, then $C_d(n)C_{d'}(n) = C_{d{d'}}(n)$ for every  $n \in \mathbb{N}$.
\item \label{Ram:power}
For any prime number $p$ and natural number $k$,
\[ C_{p^k}(n)  = \begin{cases} 0 \ \ & \text{ if } p^{k-1} \nmid n \\
-p^{k-1} \ \ & \text{ if } p^{k-1} \mid n  \ \text{ and } p^k \nmid n \\
\phi(p^k) \ \ & \text{ if } p^k \mid n  
\end{cases} \]
\end{enumerate} 

\medskip
Let $h$ be a {\it primitive element} of $\mathbb{Z}_p$ for a prime number $p$. That is, $h$ is a generator for the multiplicative group $\mathbb{Z}_p\setminus \{0\}$. It is well-known that $\mathbb{Z}_p$ has a primitive element if $p$ is prime and that $\{ h^{t} (\mod p) \colon t \in U(p-1) \}$ is the set of all primitive elements of $\mathbb{Z}_p$. Furthermore, 
\begin{theorem}[\cite{park00}]
\label{thm:primitive}
If $p$ and $q$ are prime such that $p = 4q+1$, then $2$ is a primitive element of $\mathbb{Z}_p$. 
\end{theorem}
Using this theorem, we obtain the following proposition.

\begin{proposition}
\label{pro:primitive}
If $p$ and $q$ are prime such that $p = 4q+1$, then the set of all primitive elements of $\mathbb{Z}_p$ is $ \{ 2^t (\mod\; p) \colon t \text{ is odd and } t \neq q, 3q \}$. 
\end{proposition}
\begin{proof}
By Theorem~\ref{thm:primitive}, $2$ is a primitive element of $\mathbb{Z}_p$. Since $U(p-1) = U(4q)$, a natural number $l$ is in $U(p-1)$ iff $(l, 4q) = 1$. Hence $U(p-1)$ consists of all odd numbers up to $4q$ that are not divisible by $q$.
\end{proof}

Note that in the above case the number of primitive elements of $\mathbb{Z}_p$ for prime $p$ is $ |U(p-1)|  =   2(q-1)$.

\medskip

Another result that we make use of in proving the next proposition and throughout this paper is the following straightforward identity (c.f. \cite{ZachThesis}): for an odd prime $p$,
\begin{equation}
\label{eq:relation}
(h^{k}(\mod\; p))+(h^{(p-1)/2+k}(\mod\; p))=p, \quad \hbox{ for } k=1,\ldots,(p-1)/2.
\end{equation}

\begin{proposition}
\label{pro:pri2}
If $p$ and $q$ are prime such that $p = 4q+1$, then for every primitive element $h$ of $\mathbb{Z}_p$, $\{ h^q (\mod\; p), h^{3q} (\mod\; p) \} = \{  2^q (\mod\; p), 2^{3q} (\mod\; p)  \}$.
\end{proposition}
\begin{proof}
The order of $2^q (\mod\; p)$ in the multiplicative group $\mathbb{Z}_p\setminus \{0\}$ is $4$ and the order of $2^{3q} (\mod\; p)$ is also 4. By the elementary result in group theory (\cite{gallian}), the number of elements of order 4 in the cyclic group $\mathbb{Z}_p\setminus \{0\}$ is $\phi(4) = 2$. Let $r= 2^q (\mod\; p)$. Note that $\{ r, p-r \}$ is the set of elements of order 4 since $2^{3q} (\mod\; p) = p-r$ by (\ref{eq:relation}). For every primitive element $h$, the two elements $h^q$ and $h^{3q}$ have order 4. Hence $\{ h^q (\mod\; p), h^{3q} (\mod\; p) \} = \{  2^q (\mod\; p), 2^{3q} (\mod\; p)  \}$.
\end{proof}

\section{Main result}
\label{sec:Main}
In this section, we present our main theorem and proof. For a vector $\mathbf{v} = (v_0, v_1, \cdots, v_{n-1})$ in $\mathbb{C}^n$, we define $circ\{ \mathbf{v} \}$ as the circulant matrix whose first row vector consists of $\mathbf{v}$. Our main concern is the invertibility of $circ\{\mathbf{v} \}$ for $\mathbf{v} \in \mathbb{Q}^n$. We begin with the basic theorem (\cite{kra12,HJ}) on calculation of the determinant of any circulant matrix. 


\begin{theorem}[\cite{kra12,HJ}]
\label{thm:ams}
Let $\epsilon$ be a primitive $n$-th root of unity, $\mathbf{v} = (v_0, v_1, \cdots, v_{n-1}) \in \mathbb{C}^n$, and $V = circ \{\mathbf{v}\}$. Then the determinant of the circulant matrix $V$ is obtained by the following:
\[ \det V = \prod_{l = 0}^{n-1} \left(\sum_{j = 0}^{n-1} \epsilon^{jl}v_j \right) = \prod_{l = 0}^{n-1} \lambda_l,\]
where $\lambda_l := v_0+\epsilon^lv_1 + \cdots +\epsilon^{(n-1)l}v_{n-1}$, $0 \le l \le n-1$, are the eigenvalues of $V$.
\end{theorem}


For concreteness, we may take $\epsilon$ as $e^{2\pi i/n}$ in the above theorem.
From the theorem, it is easy to see that the matrix $circ\{(v_0, v_1, \cdots, v_{n-1})\}$ is nonsingular if there exists $j$ such that $|v_j|>\sum_{k\ne j}|v_k|$ (cf. \cite{HJ}). This sufficient condition is very restrictive because the magnitude of one entry dominates, although it applies when a circulant matrix consists of complex numbers. In general, entries in a row are rather ``{\it balanced} '' meaning that an entry usually has a smaller magnitude than the sum of the magnitude for the rest. 

For example, for our matrix $A_{p,m}$ in (\ref{eq:Apm}), the corresponding circulant matrix $G_{p,m}$ will satisfy the above sufficient condition exactly when \((p-1)^{m}-\sum_{k=1}^{p-2}k^{m}>0\), and this is the statement of Theorem~\ref{thm:Apm} (although a different proof is used in \cite{ZachThesis}). In fact, since $m\geq \log(p-2)/\log((p-1)/(p-2))$ implies \((p-1)^{m}-\sum_{k=1}^{p-2}k^{m}>0\) for any fixed odd prime, from this theorem, we already know that $A_{p,m}$ is invertible when $m\geq \log(p-2)/\log((p-1)/(p-2))$. Thus, what is still left unresolved is whether the matrix $A_{p,m}$ is invertible or not, when $m<\log(p-2)/\log((p-1)/(p-2))$. Our theorems in Section~\ref{subS:GMD} provide answers for this balanced case.

Another example we consider in this paper is the family of circulant $0, 1$-matrices, and these  matrices are always ``balanced'' except for the trivial case when the entry $1$ appears only once in a row, which corresponds to the identity matrix. Thus, our theorem in Section~\ref{subS:01circulant} answers to the question of the invertibility of  $0,1$-matrices for the balanced case.

\medskip

Our main theorem covers the most general cases known so far in terms of matrix size.
We only need a restriction for entries to be rational, and throughout the paper, we assume that circulant matrices have entries in $\mathbb{Q}$ unless otherwise mentioned. 

\begin{theorem}
\label{thm:main}
For a positive number $n$, let $V$ be an $n \times n$ circulant matrix with the first row $(v_0, v_1, \cdots, v_{n-1}) \in \mathbb{Q}^n$. If $V$ is singular, then either the entries in the first row must satisfy $\sum_{j = 0}^{n-1} v_j = 0$ or there exists $d \mid n$ and $d \neq 1$ such that 
\begin{equation}
\label{eq:maincond}
C_d(0)v_0 +C_d(1)v_1 + C_d(2)v_2 + \cdots + C_d(n-1)v_{n-1} = 0,
\end{equation}
where $C_d(k)$ is the Ramanujan's sum and $C_d(0) = \phi(d)$ for convention.
\end{theorem}

\begin{proof}
Suppose that $\det V = 0$. By Theorem~\ref{thm:ams}, there exists $l$ with $0 \le l \le n-1$ such that $\lambda_l = 0$.
In other words, for some $l$, 
\begin{equation}
\label{eq:lambda_l}
\lambda_l = v_0+\epsilon^lv_1 +\epsilon^{2l}v_2 + \cdots \epsilon^{(n-1)l}v_{n-1} = 0.
\end{equation}


By Theorem~\ref{thm:lang}, we know that $\epsilon^l$ is a primitive $d$-th root of unity for some $d \mid n$ and $\epsilon^l$ is a root of $\Phi_d(x) = 0$. If $d=1$, then $\epsilon^l=1$ and the identity (\ref{eq:lambda_l}) becomes $\sum_{j = 0}^{n-1} v_j = 0$. We now assume $d\neq 1$ and let $w := \epsilon^l$. Let $w_1 := w, w_2, w_3, \cdots, w_{\phi(d)}$ be all the (distinct) roots of $\Phi_d(x) = 0$. By Proposition~\ref{prop:cy}, for each $i$, there exists $f_i$ on $\mathbb{Q}(w)$ such that $f_i(w) = w_i$ and $f_i\mid_{\mathbb{Q}} = id$. From the identity (\ref{eq:lambda_l}) we have
\[ v_0 + wv_1 + w^2v_2 + \cdots + w^{n - 1}v_{n-1} = 0. \] 
By applying $f_i$ to this identity for every $i = 2, 3, \cdots, \phi(d)$, we get
\begin{eqnarray*}
0=f_i(0)&=&f_i(v_0 + wv_1 + w^2v_2 + \cdots + w^{n - 1}v_{n-1}) \\
&=& v_0 + f_i(w)v_1+(f_i(w))^2v_2 + \cdots + (f_i(w))^{n-1}v_{n-1}\\
&=& v_0 + w_iv_1 + w_i^2v_2 + \cdots + w_i^{n - 1}v_{n-1}
\end{eqnarray*}
Therefore, we obtain the following list of identities; 
\begin{eqnarray*}
v_0 + w_1v_1 + w_1^2v_2 + &\cdots& + w_1^{n - 1}v_{n-1} = 0  \\
v_0 + w_2v_1 + w_2^2v_2 + &\cdots& + w_2^{n - 1}v_{n-1} = 0  \\
 &\cdots&\\
v_0 + w_{\phi(d)}v_1 + w_{\phi(d)}^2v_2 + &\cdots& + w_{\phi(d)}^{n - 1}v_{n-1} = 0  \end{eqnarray*}
We add the $\phi(d)$ identities to obtain the following:
\[ \phi(d)v_0 + \left(\sum_{j = 1}^{\phi(d)}w_j\right)v_1 + \left(\sum_{j = 1}^{\phi(d)}w_j^2\right)v_2 + \cdots + \left(\sum_{j = 1}^{\phi(d)}w_j^{n-1}\right)v_{n-1} = 0 \]
Note that  $\sum_{j = 1}^{\phi(d)}w_j^k$, the coefficient of $v_k$, is the  Ramanujan's sum
\[C_d(k) = \sum_{\genfrac{}{}{0pt}{}{1\le a \le d}{(a, d) = 1}} e^{2\pi i \frac{a}{d}k} \ \ . \] 
\end{proof} 

\medskip
\begin{remark}
 We consider $x^n - 1 = 0 $ and a primitive $n$-th root of unity $\epsilon$. In the above theorem and throughout the paper, for convenience, we may let $\epsilon = e^{2\pi i/n}$, but there are exactly $\phi(n)$ many primitive $n$-th root of unity, and every primitive root makes an equivalent argument throughout our reasoning in our main theorem. 
 \end{remark}

For obtaining the complete list of constraints of the condition as a necessary condition to be singular, it is crucial to calculate Ramanujan's sums as the coefficients of a linear combination of the entries of the first row. 

Our main theorem states a necessary condition for a circulant matrix to be singular. Equivalently, as a corollary, we can restate a sufficient condition for a circulant matrix $circ\{(v_0, v_1, \cdots, v_{n-1})\}$ to be nonsingular.

\begin{corollary}
\label{cor:main}
Let $V$ be an $n \times n$ circulant matrix with the first row $(v_0, v_1, \cdots, v_{n-1}) \in \mathbb{Q}^n$ satisfying $\sum_{j = 0}^{n-1} v_j \ne 0$. If, for every $d \mid n$ with $d\ne 1$,
\begin{equation}
\label{eq:maincorcond}
C_d(0)v_0 +C_d(1)v_1 + C_d(2)v_2 + \cdots + C_d(n-1)v_{n-1}\ne 0,
\end{equation}
then $V$ is nonsingular. 
\end{corollary}

\begin{proof}
 Suppose to the contrary that, for every $d \mid n$ with $d\ne 1$, the identity in (\ref{eq:maincorcond}) holds true for every $d \mid n$ with $d\ne 1$, and that $V=circ \{ (v_0, v_1, \cdots, v_{n-1}) \}$ with $\sum_{j = 0}^{n-1} v_j \ne 0$ is singular. By Theorem~\ref{thm:main}, we see that there exists $d \mid n$ and $d \neq 1$ such that (\ref{eq:maincond}) holds true, which is a contradiction.
\end{proof}

\section{Explicit conditions for various matrix size $n$} 
\label{sec:Explicit}
Let $$\mathcal{E}_n = \bigcup_{\substack{d \ge 1 \\ d \mid n}}\left\{ (v_0, v_1, \cdots, v_{n-1}) \in \mathbb{Q}^n \colon  \sum_{j = 0}^{n-1}C_d(j)v_j = 0 \right\}.$$ Using this notation, Corollary~\ref{cor:main} is restated that if $\mathbf{v} \notin \mathcal{E}_n$ then $circ\{\mathbf{v}\}$ is nonsingular.  Therefore, determining $\mathcal{E}_n$ completely and explicitly is a key for the invertibility of rational circulant matrices, which is done in this section for various $n$. 
To determine $\mathcal{E}_n$ for a given size $n$ we consider a prime factorization of $n = q_1^{e_1}q_2^{e_2} \cdots q_{t}^{e_t}$, where $q_i$'s are distinct primes. In particular, we will obtain the conditions explicitly for infinitely many $n$.
\medskip

The invertibility of rational circulant matrices is known only for when $n$ is prime (Proposition 23, \cite{kra12}). First of all, using our result in this paper we will prove this statement for prime $n$ as a special case. Let $\mathbf{v} = (v_0, v_1, \cdots, v_{q-1})$, where $q$ is prime. We consider $circ\{ \mathbf{v} \}$.
We obtain $ \phi(q)v_0 + C_q(1)v_1 + C_q(2)v_2 + \cdots + C_{q}(q-1)v_{q-1} = 0$. Therefore,
\[(q-1)v_0 = v_1+v_2+\cdots + v_{q-1}. \] From above, we also have $v_0 + v_1 + \cdots + v_{q-1} = 0$. So 
\[\mathcal{E}_q = \left\{\mathbf{v} \in \mathbb{Q}^n \colon \sum_{j = 0}^{q-1}v_j = 0 \right\} \bigcup \left\{\mathbf{v} \in \mathbb{Q}^n \colon(q-1)v_0 = \sum_{j =1}^{q-1} v_j \right\}. \] 

\begin{corollary}(\cite{kra12})
\label{cor:q}
Suppose that $q$ is prime. The matrix $circ\{(v_0, v_1, \cdots, v_{q-1})\}$ is singular iff either $\sum_{j = 0}^{q-1}v_j = 0$ or  $v_0 = v_1 = \cdots = v_{q-1}$. 
\end{corollary}

\begin{proof}

Note that the proof for sufficiency of being singular is trivial. Suppose that the matrix $circ\{ (v_0, v_1, \cdots, v_{q-1}) \}$ is singular. If $\sum_{j = 0}^{q-1}v_j \neq 0$, we need to show that all $v_j$'s are equal. Now assume that not all $v_j$'s are equal. By letting $v_{i_0} = \max_{0 \le i \le q-1 } \{ v_i \}$ and $v_{i_1} = \min_{0 \le i \le q-1} \{ v_i\}$, we have $v_{i_0} \ge v_i$ for all $i$ and in particular $v_{i_0} > v_{i_1}$. Therefore, $(q-1)v_{i_0} > \sum_{j \neq i_0} v_j$. Let $\mathbf{u}:=(v_{i_0}, v_{i_0+1}, \cdots, v_{q+i_0-1})$, where the subscripts are modulo $q$. Then, since $\mathbf{u}\notin \mathcal{E}_q$, by Corollary~\ref{cor:main}, the matrix $circ\{\mathbf{u}\}$ is nonsingular.
Since $\mathbf{u}$ is simply a circular rotation of the row vector $(v_0, v_1, \cdots, v_{q-1})$, the matrix $circ\{ (v_0, v_1, \cdots, v_{q-1}) \}$ nonsingular as well, and this is a contradiction.
  \end{proof}  
  
\medskip

For any $n \ge 1$, since $1 \mid n$, we see that every solution vector $\mathbf{v} = (v_0, v_1, \cdots, v_{n-1})$ of $v_0 + v_1 + \cdots + v_{n-1} = 0$ is always included in $\mathcal{E}_n$. We let $\mathcal{E}_n^1 = \{ (v_0, v_1, \cdots, v_{n-1}) \colon \sum_{j = 0}^{n-1}v_j = 0 \}$. Now we assume that $d$ in Theorem~\ref{thm:main} is larger than 1. To determine $\mathcal{E}_n$ completely we first need to obtain all the factors of $n$ bigger than 1 and for each factor $d$ we obtain the corresponding identity (\ref{eq:maincond}). Below we explicitly write the set $\mathcal{E}_n$, consisting of the union of the solution vectors for each of these identities for various $n$. Although it is clear that for any integer $n$, we can find such a set $\mathcal{E}_n$, we only consider integers $n$ the form $n = 2^{k_1}q^{k_2}r^{k_3}$, where $q$ and $r$ are distinct odd primes and $k_i \ge 0$. 

\begin{enumerate}[leftmargin=0cm,itemindent=0cm,labelwidth=\itemindent,labelsep=0cm,align=left,label={\bf Type \Roman*}]
\item {\;\bf ($n = q^k$, $q$ odd prime and $k \ge 2$).}
\label{type:qk}
Let $d = q^t$. For each $t$ with $1\le t \le k$, we obtain
\[ C_{q^t}(0)v_0 +C_{q^t}(1)v_1 + C_{q^t}(2)v_2 + \cdots + C_{q^t}(q^k-1)v_{q^k-1} = 0. \]
By property \ref{Ram:power} about Ramanujan's sum, $C_{q^t}(j)= 0$ if $q^{t-1} \nmid j$. Also, $C_{q^t}(j)$ is positive iff $q^t \mid j$. Since $\phi(q^t)  = q^t - q^{t-1} = q^{t-1}(q-1)$, we obtain the following identity: for $1\le t \le k$,
\begin{equation}
\label{eq:qkCond}
(q-1)\sum_{j = 0}^{q^{k-t}-1}v_{q^tj} 
= \sum_{j = 0}^{q^{k-t}-1}\left( v_{q^tj+q^{t-1}}+v_{q^tj+2q^{t-1}}+ \cdots +  v_{q^tj+(q-1)q^{t-1}}\right).
\end{equation}
Thus, $\mathcal{E}_{q^k}= \bigcup_{0 \le t \le k}\Big\{ (v_0, v_1, \cdots, v_{q^k-1}) \colon \hbox{{\rm identity}} \; (\ref{eq:qkCond})  \Big\} . $

\smallskip
\item {\;\bf ($n = 2^k$, $k \ge 2$).}
\label{type:2k}
By using a similar argument as in \ref{type:qk} for $2$ instead of odd prime $q$, we see that
\[ \mathcal{E}_{2^k} = \bigcup_{0 \le t \le k}\left\{\mathbf{v} \colon \sum_{j = 0}^{2^{k-t}-1}v_{2^tj} 
= \sum_{j = 0}^{2^{k-t}-1}v_{2^tj+2^{t-1}} \right\} . \]

\item {\;\bf ($n = 2q$, $q$ odd prime).} 
\label{type:2q}
The possible factors are $d = 2, q$, and $2q$. 
\begin{enumerate}[label=\arabic*)]
\item $d = 2$: We obtain \[ v_0 + C_2(1)v_1 +C_2(2)v_2 + \cdots + C_2(2q-1)v_{2q-1} = 0. \] It is easy to see that $C_2(j) = 1$ if $j$ is even and $C_2(j) = -1$ if $j$ is odd. Hence, the identity is simplified as 
\[ \sum_{j = 0}^{q-1} v_{2j} = \sum_{j = 0}^{q-1} v_{2j+1}. \]
\item $d = q$: We obtain
\[ (q-1)v_0 + C_q(1)v_1 +C_q(2)v_2 + \cdots + C_q(2q-1)v_{2q-1} = 0 \] 
and we know that $C_q(j) = \begin{cases} -1 & \text{ if } q \nmid j \\ q-1 & \text{ if } q \mid j \end{cases}$. Thus the identity is simplified as
\begin{equation}
\label{eq:2qCond2} 
(q-1) (v_0+v_q)=  \sum_{j = 1}^{q-1}(v_j + v_{q+j}). 
\end{equation}
\item $d = 2q$:
The same argument as before gives the following identity:
\[ \phi(2q)v_0 + C_{2q}(1)v_1 +C_{2q}(2)v_2 + \cdots + C_{2q}(2q-1)v_{2q-1} = 0. \] 
It is known that $\phi(2q) = \phi(q) = q-1$ and $C_{2q}(j) = C_2(j)C_q(j)$. We obtain 
\begin{equation}
\label{eq:2qCond3} 
 (q-1)v_q + \sum_{j=1}^{q-1}v_{2j}= (q-1)v_0+\sum_{j = 1}^{\frac{q-1}{2}}v_{2j-1} + \sum_{j = \frac{q+1}{2}+1}^{q-1}v_{2j-1}.
\end{equation}
\end{enumerate} 
Thus, \[\mathcal{E}_{2q}= \mathcal{E}_{2q}^1 \bigcup \left\{ \mathbf{v} \colon \sum_{j = 0}^{q-1} v_{2j} = \sum_{j = 0}^{q-1} v_{2j+1}\right\}\bigcup\Big\{ \mathbf{v} \colon \hbox{{\rm identity}} \; (\ref{eq:2qCond2}) \hbox{{\rm \ or}} \; (\ref{eq:2qCond3})\Big\}.\]

\smallskip
\item {\;\bf ($n = 2q^k$, $q$ odd prime and $k \ge 2$).}
The possible factors are $d = 2$, $d = q^t, \ 1 \le t \le k$, and $d = 2q^t, \ 1 \le t \le k$. 
\begin{enumerate}[label=\arabic*)]
\item $d = 2$: We have
\[ (2-1)v_0 + C_2(1)v_1 +C_2(2)v_2 + \cdots + C_2(2q^k-1)v_{2q^k-1} = 0 \]
which is simplified as $\sum_{j = 0}^{q^k-1} v_{2j} = \sum_{j = 0}^{q^k-1} v_{2j+1}$.

\item $d = q^t$ with $1 \le t \le k$: We have 
\[ C_{q^t}(0)v_0 +C_{q^t}(1)v_1 + C_{q^t}(2)v_2 + \cdots + C_{q^t}(2q^k-1)v_{2q^k-1} = 0 \]
and by \ref{type:qk} this is simplified as
\begin{equation}
\label{eq:2qkCond2} 
(q-1)\sum_{j = 0}^{2q^{k-t}-1}v_{q^tj} 
= \sum_{j = 0}^{2q^{k-t}-1}\left( v_{q^tj+q^{t-1}}+v_{q^tj+2q^{t-1}}+ \cdots +  v_{q^tj+(q-1)q^{t-1}}\right).
\end{equation} 

\item $d = 2q^t$ with $1 \le t \le k$: We have
\[ C_{2q^t}(0)v_0 +C_{2q^t}(1)v_1 + C_{2q^t}(2)v_2 + \cdots + C_{2q^t}(2q^k-1)v_{2q^k-1} = 0.  \]
Since $C_{2q^t}(0) = \phi(2q^t) = q^{t-1}(q-1)$ and $C_{2q^t}(j) = C_2(j)C_{q^t}(j)$ we obtain

\begin{eqnarray}
\label{eq:2qkCond3}
&{\;}&(q-1)\sum_{j = 0}^{q^{k-t}-1}v_{2q^tj} + \sum_{j = 0}^{q^{k-t}-1}\left( v_{2q^tj+q^{t-1}}+v_{2q^tj+3q^{t-1}}+ \cdots +  v_{2q^tj+(q-2)q^{t-1}}\right)\nonumber\\
&+&\sum_{j = 0}^{q^{k-t}-1}\left( v_{2q^tj+q^t+2q^{t-1}}+v_{2q^tj+q^t+4q^{t-1}}+ \cdots +  v_{2q^tj+q^t+(q-1)q^{t-1}}\right) \nonumber\\
&=& (q-1)\sum_{j = 0}^{q^{k-t}-1}v_{2q^tj+q^t}\nonumber\\
&+&\sum_{j = 0}^{q^{k-t}-1}\left( v_{2q^tj+2q^{t-1}}+v_{2q^tj+4q^{t-1}}+ \cdots +  v_{2q^tj+(q-1)q^{t-1}}\right) \nonumber\\
&+&\sum_{j = 0}^{q^{k-t}-1}\left( v_{2q^tj+q^t+q^{t-1}}+v_{2q^tj+q^t+3q^{t-1}}+ \cdots +  v_{2q^tj+q^t+(q-2)q^{t-1}}\right)
\end{eqnarray}
\end{enumerate}
Thus, \[\mathcal{E}_{2q^k}= \mathcal{E}_{2q^k}^1\bigcup \left\{ \mathbf{v} \colon \sum_{j = 0}^{q^k-1} v_{2j} = \sum_{j = 0}^{q^k-1} v_{2j+1}\right\}\bigcup \left[\bigcup_{1 \le t \le k} \{  \mathbf{v} \colon \hbox{{\rm identity}} \; (\ref{eq:2qkCond2}) \hbox{{\rm \ or}} \; (\ref{eq:2qkCond3})\} \right]. \]

\smallskip
\item 
\label{type:2kq}
{\;\bf ($n = 2^kq$, $q$ odd prime and $k \ge 2$).}
\begin{enumerate}[label=\arabic*)]
\item For $d =q$, from \ref{type:qk} we obtain the following condition.
\[  (q-1)\sum_{j = 0}^{2^k-1}v_{qj} = \sum_{j = 0}^{2^k-1} (v_{qj+1}+v_{qj+2}+\cdots+v_{qj+q-1})\]

\item For $d = 2^t$ with $t \ge 1$, from \ref{type:2k} we obtain the following condition.

\begin{equation}
\label{eq:2kqCond2}
  \sum_{j = 0}^{2^{k-t}q-1}v_{2^tj} = \sum_{j = 0}^{2^{k-t}q-1}v_{2^tj+2^{t-1}} .  
\end{equation}

\item For $d = 2^tq$ with $t \ge 1$, the coefficient for $v_j$ is $C_{2^tq}(j) = C_{2^t}(j)C_{q}(j)$. Note first that the value $C_{2^t}(j)$ is zero if $2^{t-1} \nmid j$. For $j$ divisible by $2^{t-1}$, we have that (i) $C_{2^t}(j) = -2^{t-1}$ if $2^{t-1} \mid j$ and $2^t \nmid j$, and (ii) $C_{2^t}(j) = 2^{t-1}$ if $2^t \mid j$. The value $C_q(j)$ is $q-1$ if $q \mid j$ and $-1$ if $q \nmid j$. Hence we obtain
\begin{eqnarray}
\label{eq:2kqCond3}
&{\;}&(q-1)\sum_{j = 0}^{2^{k-t}-1}v_{2^tqj} \\
&+&  \sum_{j = 0}^{2^{k-t}-1}(v_{2^tqj+2^{t-1}q+2^t}+v_{2^tqj+2^{t-1}q+2\cdot2^t}+ \cdots + v_{2^tqj+2^{t-1}q+(q-1)\cdot2^t} ) \nonumber \\ 
&+& \sum_{j = 0}^{2^{k-t}-1}(v_{2^tqj+2^{t-1}}+v_{2^tqj+3\cdot2^{t-1}}+v_{2^tqj+5\cdot2^{t-1}}+\cdots + v_{2^tqj+(q-2)2^{t-1}})\nonumber\\
&+& \sum_{j = 0}^{2^{k-t}-1}( v_{2^tqj+(q+2)2^{t-1}}+ v_{2^tqj+(q+4)2^{t-1}}+\cdots + v_{2^tqj+(2q-1)2^{t-1}} )\nonumber\\
&=& (q-1)\sum_{j = 0}^{2^{k-t}-1}v_{2^tqj+2^{t-1}q}+ \sum_{j = 0}^{2^{k-t}-1}(v_{2^tqj+2^t}+v_{2^tqj+2\cdot2^t}+ \cdots + v_{2^tqj+(q-1)\cdot2^t} )\nonumber
\end{eqnarray}
\end{enumerate} 

\smallskip

Thus, 
\begin{eqnarray*}
\mathcal{E}_{2^kq}& = & \mathcal{E}_{2^kq}^1\bigcup\left[\bigcup_{1 \le t \le k} \{  \mathbf{v} \colon \hbox{{\rm identity}} \; (\ref{eq:2kqCond2}) \hbox{{\rm \ or}} \; (\ref{eq:2kqCond3})\} \right] \\
& & \bigcup \left\{ \mathbf{v} \colon  (q-1)\sum_{j = 0}^{2^k-1}v_{qj} = \sum_{j = 0}^{2^k-1} (v_{qj+1}+v_{qj+2}+\cdots+v_{qj+q-1}) \right\}. 
\end{eqnarray*}
\smallskip

\item {\;\bf ($n = 2qr$, $q$ and $r$ odd primes with $q < r$).} 
With the possible factors $d = 2, q, r, 2q, 2r, qr$ and $2qr$, we use similar arguments as in \ref{type:2q} and obtain the following list of identities.

\begin{enumerate}[label=\arabic*)]
\item $d = 2$: 
\[ \sum_{j = 0}^{qr-1} v_{2j} = \sum_{j = 0}^{qr-1} v_{2j+1} \]
\item 
\begin{enumerate}
\item $d = q$: 
\begin{equation}
\label{eq:2qrCond2_1}
(q-1)\sum_{j = 0}^{2r-1} v_{jq} = \sum_{i = 0}^{2r-1} \sum_{j = 1}^{q-1} v_{iq+j}
\end{equation}
\item $d = r$:
\begin{equation}
\label{eq:2qrCond2_2}
(r-1)\sum_{j = 0}^{2q-1} v_{jr} = \sum_{i = 0}^{2q-1} \sum_{j = 1}^{r-1} v_{ir+j}
\end{equation}
\end{enumerate} 
\item 
\begin{enumerate}
\item $d = 2q$: 
\begin{equation}
\label{eq:2qrCond3_1}
\sum_{\substack{a=1\\ q \nmid a}}^{qr-1} v_{2a} + (q-1)\sum_{b=0}^{r-1} v_{(2b+1)q} = \sum_{\substack{a=0\\ q \nmid(2a+1)}}^{qr-1}v_{2a+1} + (q-1) \sum_{b=0}^{r-1} v_{2bq}  
\end{equation}
\item $d = 2r$:  
\begin{equation}
\label{eq:2qrCond3_2}
\sum_{\substack{a=1\\ r \nmid a}}^{qr-1} v_{2a} + (r-1)\sum_{b=0}^{q-1} v_{(2b+1)r} = \sum_{\substack{a=0\\ r \nmid(2a+1)}}^{qr-1}v_{2a+1} + (r-1) \sum _{b=0}^{q-1}v_{2br}
\end{equation}
\end{enumerate}

\item $d = qr$: From property \ref{Ram:prod} in the Ramanujan's sum,  $C_{qr}(j) = C_q(j)C_r(j)$ if $(q, r) = 1$.  
\begin{eqnarray}
\label{eq:2qrCond4}
	&{\;}&(q-1)(r-1)(v_0+v_{qr}) + \sum_{\substack{i=1\\ q, r \nmid i}}^{2qr-1} v_i \nonumber\\
	&{\;}&=\;(q-1)(v_q + v_{2q} + \cdots v_{q(r-1)} + v_{q(r+1)}+ \cdots + v_{q(2r-1)}) \nonumber\\
	&{\;}&+\;\;(r-1)(v_r + v_{2r} + \cdots v_{r(q-1)} + v_{r(q+1)}+ \cdots + v_{r(2q-1)})
\end{eqnarray}
\item $d  = 2qr$: 
\begin{eqnarray}
\label{eq:2qrCond5}
	&{\;}&(q-1) \sum_{\genfrac{}{}{0pt}{}{j = 0}{j \neq \frac{r-1}{2}}}^{r-1} v_{(2j+1)q} + (r-1)\sum_{\genfrac{}{}{0pt}{}{j = 0}{j \neq \frac{q-1}{2}}}^{q-1} v_{(2j+1)r} + (q-1)(r-1)v_{0} + \sum_{\genfrac{}{}{0pt}{}{i = 1}{q, r \nmid i}}^{qr-1} v_{2i} \nonumber\\
	&{\;}&=\;(q-1)\sum_{j = 1}^{r-1} v_{2jq} + (r-1)\sum_{j = 1}^{q-1}v_{2jr}\nonumber\\ 
	&{\;}&+ (q-1)(r-1)v_{qr} + \sum_{\genfrac{}{}{0pt}{}{i = 0}{q, r \nmid (2i+1)}}^{qr-1} v_{2i+1}
\end{eqnarray}
\end{enumerate}
Thus, 
\begin{eqnarray*}
\mathcal{E}_{2qr}& = & \mathcal{E}_{2qr}^1\bigcup \left\{ \mathbf{v} \colon \sum_{j = 0}^{qr-1} v_{2j} = \sum_{j = 0}^{qr-1} v_{2j+1}\right\}\\
& & \bigcup \Big\{  \mathbf{v} \colon \hbox{{\rm identity}} \; (\ref{eq:2qrCond2_1}),  (\ref{eq:2qrCond2_2}), (\ref{eq:2qrCond3_1}), (\ref{eq:2qrCond3_2}), (\ref{eq:2qrCond4}) \hbox{{\rm \ or}} \; (\ref{eq:2qrCond5}) \Big\}.
\end{eqnarray*}
\end{enumerate}

\section{Applications: Determining the invertibility of circulant matrices} 
\label{sec:Appl}
\subsection{Application to the generalizations of the Maillet determinant}
\label{subS:GMD}

In this subsection, we use our main results in this paper in determining the invertibility of the matrices $A_{p,m}$ defined in (\ref{eq:Apm}).

When $m=1$, the leading principal minor of $A_{p,1}$ of order $(p-1)/2$ is called the {\it Maillet determinant} and studied extensively in the literature \cite{TM,HWT,CO,CO2}. In particular, it is shown that the Maillet determinant does not vanish for any $p\ge 3$.

It is easy to see that when $m=1$, the determinant of $A_{p,1}$ vanishes for every $p\ge 5$, and when $p=3$, the determinant of $A_{3,m}$ does not vanish for every $m\ge 1$ (c.f. \cite{GMD,ZachThesis}). Therefore in this paper, we investigate the invertibility of $A_{p,m}$ only for the cases where $m\ge 2$ and $p\ge 5$. For the first few choices of \(p\), the matrices \(A_{p,m}\) are:
$$
A_{5,m}=\left[\begin{array}{cccc} 
1&2^m&3^m&4^m\\ 
3^m&1&4^m&2^m\\
2^m&4^m&1&3^m\\
4^m&3^m&2^m&1
\end{array}\right],\quad
A_{7,m}=\left[\begin{array}{cccccc} 
1&2^m&3^m&4^m&5^m&6^m\\ 
4^m&1&5^m&2^m&6^m&3^m\\
5^m&3^m&1&6^m&4^m&2^m\\
2^m&4^m&6^m&1&3^m&5^m\\
3^m&6^m&2^m&5^m&1&4^m\\
6^m&5^m&4^m&3^m&2^m&1
\end{array}\right]
$$

Many properties of $A_{p,m}$ are studied in \cite{GMD,ZachThesis}. However, the question of for which values of $m$ and $p$ the matrix $A_{p,m}$ is nonsingular is not answered completely for the most part, although some partial results are obtained there. We will apply our results obtained in this paper to see how we can provide additional information to answer this question. Below we list some results that are the most relevant to our discussion and refer \cite{GMD,ZachThesis} for more details.

\begin{theorem}[\cite{GMD}]
\label{thm:Apm}
The matrix $A_{p,m}$ is nonsingular if \((p-1)^{m}-\sum_{k=1}^{p-2}k^{m}>0\). In particular, for any odd $p$, $A_{p,m}$ is nonsingular if $m\geq \log(p-2)/\log((p-1)/(p-2))$.
\end{theorem}

In the proof of the above theorem in \cite{GMD}, a simple observation that the matrix $JA_{p,m}$ is diagonally dominant where $J$ is the $(p-1)$ by $(p-1)$ reversal matrix is used, without invoking the next result which says $A_{p,m}$ is similar to a circulant matrix.

\begin{theorem}[\cite{GMD,ZachThesis}]
\label{thm:Gpm}
Let $h$ be a primitive of $\mathbb{Z}_{p}$. Then $A_{p,m}$ is permutation similar to a circulant matrix 
$$G_{p,m}:=G_{p,m,h}:=circ\{(1,(h(\mod\; p))^m, (h^2(\mod\; p))^m, \dots,(h^{p-2}(\mod\; p))^m)\}$$
\end{theorem}

We recall that a primitive element $h$ of $\mathbb{Z}_{p}$ satisfies $\{h^{k}(\mod\; p):1\leq k\leq p-1\}=\{1,2,\ldots,p-1\}$. When $p=5$, there are two primitive elements $h=2$ and $h=3$ of $\mathbb{Z}_{5}$, and when $p=7$, there are two primitive elements $h=3$ and $h=5$ of $\mathbb{Z}_{7}$. Thus, Theorem~\ref{thm:Gpm} says that $A_{5,m}$ and $A_{7,m}$ are permutation similar to 
$$
G_{5,m,2}=\left[\begin{array}{cccc} 
1&2^m&4^m&3^m\\ 
3^m&1&2^m&4^m\\
4^m&3^m&1&2^m\\
2^m&4^m&3^m&1
\end{array}\right],\quad
G_{7,m,3}=\left[\begin{array}{cccccc} 
1&3^m&2^m&6^m&4^m&5^m\\ 
5^m&1&3^m&2^m&6^m&4^m\\
4^m&5^m&1&3^m&2^m&6^m\\
6^m&4^m&5^m&1&3^m&2^m\\
2^m&6^m&4^m&5^m&1&3^m\\
3^m&2^m&6^m&4^m&5^m&1
\end{array}\right]$$
respectively. 

\medskip

We now present how the results in the previous sections can be used to provide information about the invertibility of $G_{p,m}$, hence that of $A_{p,m}$.

\begin{theorem} 
\label{thm:2q}
Let $p$ be a prime of the form $p=2q+1$, with $q$ an odd prime. Then for all $m\ge 2$, $A_{p,m}$ is nonsingular.
\end{theorem}

\begin{proof}
Suppose that $A_{p,m}$ satisfies the given assumptions. By Theorem~\ref{thm:Gpm}, $A_{p,m}$ is similar to $G_{p,m} = circ\{(v_0,v_1, \cdots, v_{2q-1})\}$, where $v_j := (h^j(\mod\; p))^m$, $j=0,1,\dots, 2q-1$, for some primitive $h$ of $\mathbb{Z}_{p}$.
Then $v_0=1$, $v_q=(2q)^m$ and $v_0, v_1, \cdots, v_{2q-1}$ are all positive, hence $\sum_{j = 0}^{2q-1}v_j > 0$. Thus, to conclude that $G_{p,m}$ (hence $A_{p,m}$) is nonsingular, by Corollary~\ref{cor:main}, it suffices to show that $v_0, v_1, \cdots, v_{2q-1}$ do not satisfy any of the nontrivial identities in $\mathcal{E}_{2q}$ of \ref{type:2q} in Section~\ref{sec:Explicit}: $\sum_{j = 0}^{q-1} v_{2j} = \sum_{j = 0}^{q-1} v_{2j+1}$, (\ref{eq:2qCond2}), and (\ref{eq:2qCond3}).

To show the first identity, $\sum_{j = 0}^{q-1} v_{2j} = \sum_{j = 0}^{q-1} v_{2j+1}$, does not hold true, we observe that $\{h^j(\mod\; p): j=0,1,\dots, 2q-1\}=\{1,2,\dots, 2q\}$, and that there are exactly $q$ many odd integers in $\{v_0, v_1, \cdots, v_{2q-1}\}$. Let $\mathfrak{o}_0$ and $\mathfrak{o}_1$ be the number of odd integers in $\{v_{2j}:j=0,1,\cdots,q-1\}$ and $\{v_{2j+1}:j=0,1,\cdots,q-1\}$, respectively. Since $\mathfrak{o}_0+\mathfrak{o}_1$ is an odd number (i.e. $q$), $\mathfrak{o}_0$ and $\mathfrak{o}_1$ cannot be the same. Therefore, the two summations $\sum_{j = 0}^{q-1} v_{2j}$ and $\sum_{j = 0}^{q-1} v_{2j+1}$ in the first identity cannot be the same.


For the next identity, observe that (\ref{eq:relation}) gives $h^{q+j}(\mod\; p)= 2q+1 - (h^{j}(\mod\; p))$ for $j = 1, 2, \cdots, q-1$. Combining this with the fact that 
$\{h^j(\mod\; p), h^{q+j}(\mod\; p): j=1,\dots, q-1\}=\{1,2,\dots, 2q\}\bks \{1,2q\}$, we see that the right-hand side of (\ref{eq:2qCond2}) becomes 
$\sum_{j = 1}^{q-1}(v_j + v_{q+j})=\sum_{j = 1}^{q-1}(j+1)^m+(2q-j)^m$. Noting that $v_0+v_q=1+(2q)^m$, since for $m \ge 2$ the function $x^m + (M-x)^m$ on $1 \le x \le \frac{M}{2}$ is maximized when $x = 1$ and strictly decreasing on the interval, by considering $M=2q+1$, we get $v_0+v_q > v_j + v_{q+j}$ for every $ 1 \le j \le q-1$. Therefore, 
\[(q-1) (v_0+v_q) > \sum_{j = 1}^{q-1}(v_j + v_{q+j}),\]
and the equality in (\ref{eq:2qCond2}) does not hold true.

Finally,  since $v_{2j} > 1$ for any $j = 1, 2, \cdots, q-1$, from the above inequality we have
\begin{eqnarray*}
	(q-1)v_q + \sum_{j=1}^{q-1}v_{2j} &>& (q-1)(v_q + v_0) >  \sum_{j = 1}^{q-1}(v_j + v_{q+j}) \\
	&> & (q-1)v_0+\sum_{j = 1}^{\frac{q-1}{2}}v_{2j-1} + \sum_{j = \frac{q+1}{2}+1}^{q-1}v_{2j-1}
\end{eqnarray*}

Thus the identity (\ref{eq:2qCond3}) does not hold true.
\qquad\end{proof}

From Theorem~\ref{thm:2q}, we see that, for example, when $p=7, 11, 23, 47, 59, 83$, $A_{p,m}$ is nonsingular for all $m\ge 2$.

\begin{theorem} 
\label{thm:4q}
Let $p$ be a prime of the form $p=4q+1$, with $q$ an odd prime, and let $r= 2^q (\mod\; p)$. If either $r  = 0(\mod\; 4)$ or $r= 1 (\mod\; 4)$, then for all odd $m\ge 3$, $A_{p,m}$ is nonsingular. 
\end{theorem}

\begin{proof}
We proceed similar to the proof of Theorem~\ref{thm:2q}. 
Suppose that $A_{p,m}$ satisfies the given assumptions. By Theorem~\ref{thm:Gpm}, there exists a circulant matrix $G_{p,m} = circ\{(v_0,v_1, \cdots, v_{4q-1})\}$, with $v_j := (h^j(\mod\; p))^m$, $j=0,1,\dots, 4q-1$, similar to $A_{p,m}$, for some primitive $h$ of $\mathbb{Z}_{p}$.
Noting that $v_0=1$, $v_{2q}=(4q)^m$ and $v_0, v_1, \cdots, v_{4q-1}$ are all positive, we see $\sum_{j = 0}^{4q-1}v_j > 0$. Thus, by Corollary~\ref{cor:main}, we can conclude that $G_{p,m}$ (hence $A_{p,m}$) is nonsingular if we show that $v_0, v_1, \cdots, v_{4q-1}$ do not satisfy any of the following nontrivial identities in $\mathcal{E}_{4q}$ (c.f. \ref{type:2kq} in Section~\ref{sec:Explicit}):
\begin{enumerate}[label=(\roman*)]
\item \label{item:cond1}
$\sum_{j = 0}^{2q-1} v_{2j} = \sum_{j=0}^{2q-1}v_{2j+1}$
\item \label{item:cond2}
$\sum_{j = 0}^{q-1}v_{4j} = \sum_{j = 0}^{q-1} v_{4j+2} $
\item \label{item:cond3}
$(q-1)(v_0+v_q+v_{2q}+v_{3q}) = \sum_{j=1}^{q-1}(v_j+v_{q+j}+v_{2q+j}+v_{3q+j})$
\item \label{item:cond4} 
$(q-1)(v_0+v_{2q}) + \sum_{j=0}^{(q-3)/2}(v_{2j+1} +v_{2q+2j+1})+ \sum_{j=(q+1)/2}^{q-1}(v_{2j+1} +v_{2q+2j+1}) $\\ $= (q-1)(v_q+v_{3q}) + \sum_{j=1}^{q-1}(v_{2j} +v_{2q+2j})$
\item \label{item:cond5} 
$(q-1)v_0+\sum_{j=0}^{(q-3)/2} v_{4j+2}+\sum_{j=(q+1)/2}^{q-1}v_{4j+2}=(q-1)v_{2q} + \sum_{j = 1}^{q-1} v_{4j}$
\end{enumerate}

To show that the identity \ref{item:cond1} does not hold true, we observe that for the set $\{h^j(\mod\; p): j=0,1,\dots, 4q-1\}=\{1,2,\dots, 4q\}$, there are exactly $q$ elements in each of the four residue classes modulo 4. 
Note that $m\ge 3$ is odd. 
Since odd power of each element in an odd residue class belongs to the same class, the sum $\sum_{j = 0}^{4q-1}v_j=1^m + 2^m + \cdots + (4q)^m$ is divisible by 4.
Since odd power of each element in an odd residue class belongs to the same class, the sum $\sum_{j = 0}^{4q-1}v_j=1^m + 2^m + \cdots + (4q)^m$ is divisible by 4.
Suppose that the identity \ref{item:cond1} holds true. Let $\mathfrak{r}_1$ and $\mathfrak{r}_3$ be the size of $\{v_{2j}=1 (\mod\; 4):j=0,1,\cdots,2q-1\}$ and $\{v_{2j}=3 (\mod\; 4):j=0,1,\cdots,2q-1\}$, respectively. Then $\mathfrak{r}_1+\mathfrak{r}_3=q$, and the size of $\{v_{2j+1}=1 (\mod\; 4):j=0,1,\cdots,2q-1\}$ and $\{v_{2j+1}=3 (\mod\; 4):j=0,1,\cdots,2q-1\}$ have to be $q-\mathfrak{r}_1$ and $q-\mathfrak{r}_3$, respectively.  Note that $\mathfrak{r}_1+3\mathfrak{r}_3$ and $q-\mathfrak{r}_1+3(q-\mathfrak{r}_3) = 4q-(\mathfrak{r}_1+3\mathfrak{r}_3)$ are odd. However, each of these numbers must be even, since their sum is divisible by 4 and since the two summations $\sum_{j = 0}^{2q-1} v_{2j}$ and $\sum_{j=0}^{2q-1}v_{2j+1}$ in \ref{item:cond1} are assumed to be the same. This is a contradiction, hence we conclude that the identity \ref{item:cond1} does not hold true. 

For the next identity, note that $(h^{j'}(\mod\; p)) = 4q+1-(h^{j''} (\mod \; p))$ for any $j' = 2q+j''$ (c.f.~(\ref{eq:relation})). 
Hence for any integer $1\le a\le 2q$, the terms $a^m$ and $ (4q+1-a)^m$ either appear simultaneously or none of them appears in the sums $\sum_{j = 0}^{q-1}v_{4j}$ and $\sum_{j = 0}^{q-1} v_{4j+2} $ of the identity \ref{item:cond2}.  Also note that, in the pair $\{a^m,(4q+1-a)^m\}$, exactly one of them is even and the other is odd. It means that there are exactly $q$ odd terms in $v_{2j}$'s. Therefore, 
the two sums 
cannot be the same and the identity  \ref{item:cond2} cannot hold.

We can rewrite the identity \ref{item:cond3} as $q(v_0+v_q+v_{2q}+v_{3q}) = \sum_{j=0}^{4q-1} v_j$. Since $\sum_{j = 0}^{4q-1}v_j=\sum_{j=1}^{4q}j^m$, we see that $\sum_{j = 0}^{4q-1}v_j<\int_1^{2q+1}x^m + (4q+1-x)^m\ dx$. We claim that this integral is strictly less than $q(v_0+v_q+v_{2q}+v_{3q})$. 
To see this, note that 
\begin{eqnarray*}
&{\;}&q(v_0+v_q+v_{2q}+v_{3q})-\int_1^{2q+1}x^m + (4q+1-x)^m\ dx\\
&{\;}&=q[1+(4q)^m+a^m+(4q+1-a)^m]-\frac{(2q+1)^{m+1} - (2q)^{m+1} - 1 + (4q)^{m+1}}{m+1}
\end{eqnarray*} 
for some integer $1 < a < 4q$. 
Since $m \ge 3$, $4^mq^{m+1} \ge \frac{4^{m+1}}{m+1}q^{m+1}$. Also, it is easy to see that $\max\{a, 4q+1-a\} \ge 2q+1$ since if $a < 2q+1$ and $4q+1-a < 2q+1$ then $2q < a < 2q+1$, which is a contradiction.
Hence, either $a \ge 2q+1$ or $4q+1-a \ge 2q+1$, say $a \ge 2q+1$. Then $qa^m - \frac{1}{m+1}(2q+1)^{m+1} \ge q(2q+1)^m - \frac{2q+1}{m+1}(2q+1)^m > 0$ since $q > \frac{2q+1}{m+1}$. Similarly, we get $q(4q+1-a)^m - \frac{1}{m+1}(2q+1)^{m+1} > 0$ if $4q+1-a \ge 2q+1$. In conclusion, $q[(4q)^m + a^m + (4q+1-a)^m]- \frac{1}{m+1}[(2q+1)^{m+1} + (4q)^{m+1}]$ is already positive, hence $q(v_0+v_q+v_{2q}+v_{3q})>\sum_{j=0}^{4q-1} v_j$.


The identity \ref{item:cond5} does not hold, since $v_{2q} + v_{4j} > v_{2q}+v_0 > v_{4j+2}+v_0$ for every $j$.

We use the assumption on $r$, which has not yet been used, for the identity \ref{item:cond4}. It is known that $\sum_{j = 1}^{N}j^m$ is divisible by ${N(N+1)}/{2}$ when $m$ is odd (\cite{macmillan12}). In particular, if $N = 4q$, then $\sum_{j = 1}^{4q}j^m$ is divisible by ${4q(4q+1)}/{2}= 2qp$. Also from the argument for the identity \ref{item:cond1}, we know that $\sum_{j = 1}^{4q}j^m$ is divisible by $4$, so it is divisible by $4pq$. By letting $\alpha = \sum_{j = 0}^{2q-1} v_{2j+1}$ and $\beta=\sum_{j = 0}^{2q-1}v_{2j}$ we have $\alpha+\beta = \sum_{j = 1}^{4q}j^m$. Since $\alpha$ and $\beta$ are both odd, using again the fact that $\sum_{j = 1}^{4q}j^m$ is divisible by $4$, we know that $\alpha (\mod\; 4) \neq \beta (\mod\; 4)$ and $\alpha$ and $\beta$ are both odd. In other words, $\alpha (\mod\; 4) = 1$ and $\beta (\mod\; 4) = 3$ or vice versa. 

We rewrite the identity \ref{item:cond4} as follows: $q(v_0+v_{2q}) + \sum_{j = 0}^{2q-1} v_{2j+1} = q(v_q+v_{3q}) + \sum_{j = 0}^{2q-1}v_{2j}$. In other words, the identity becomes $q[1+(4q)^m] + \alpha = q[r^m+(p-r)^m] + \beta$. By Proposition~\ref{pro:primitive} and Proposition~\ref{pro:pri2}, the sum $\sum_{j =0}^{2q-1} v_{2j+1}$ is the sum of the $m$-th power of all the primitive elements, $r$, and $p-r$. By the given condition, we have that either $r  = 0(\mod\; 4)$ or $r= 1 (\mod\; 4)$. If $r$ is divisible by 4, then $p-r = 4q+1-r$ is congruent to 1 modulo $4$. Hence $r^m$ is a multiple of 4 and $(p-r)^m$ is congruent to 1 modulo $4$. So the left-hand side in the above condition is congruent to $(q+\alpha) (\mod\; 4)$ and the right-hand side is congruent to $(q+\beta) (\mod\; 4)$, but this cannot happen since $\alpha (\mod\; 4) \neq \beta (\mod\; 4)$. Hence the identity \ref{item:cond4}  cannot hold true. If $r$ is congruent to 1 modulo $4$, then $(p-r)^m$ is a multiple of 4. By the same argument as the previous case, the identity \ref{item:cond4} cannot hold.  
\end{proof}

\begin{table}
\centering
\begin{tabular}{c|c|c|c|c|c|c|c|c|c}
$14$&$\diamond$&$\diamond$&  $\star$&&  & $\star$  & & & $\star$\\ \hline
$13$&$\diamond$&$\diamond$&$\star$&$\star\star$&  & $\star$  & $\star\star$&  & $\star$\\ \hline
$12$&$\diamond$&$\diamond$&$\star$& &  & $\star$  & & &  $\star$\\ \hline
$11$&$\diamond$& $\diamond$& $\star$&$\star\star$&  &$\star$&$\star\star$&&  $\star$\\ \hline
$10$&$\diamond$&$\diamond$&$\star$&  &  & $\star$& & &$\star$\\ \hline
$9$ &$\diamond$&$\diamond$& $\star$&$\star\star$&  &$\star$&$\star\star$&& $\star$\\ \hline
$8$ &$\diamond$&$\star$&$\star$& & &$\star$&  && $\star$\\ \hline
$7$ &$\diamond$&$\star$&$\star$&$\star\star$& &$\star$&$\star\star$&&  $\star$\\ \hline
$6$ &$\diamond$& $\star$& $\star$ & &   & $\star$&&& $\star$\\ \hline
$5$ &$\diamond$& $\star$&$\star$&$\star\star$ &   &$\star$&$\star\star$&  & $\star$\\ \hline
$4$ &$\diamond$&$\star$&$\star$&  &   &$\star$& & & $\star$\\ \hline
$3$ & &$\star$ &$\star$&$\star\star$ & &$\star$&$\star\star$& &$\star$\\ \hline
$2$ & &$\star$&$\star$& &   &$\star$& & &$\star$\\ \hline
\slashbox{$m$}{$p$}&$5$&$7$&$11$&$13$&$17, 19$&$23$&$29$&$31, 37, 41, 43$&$47$\\ 
\end{tabular}
\medskip
\caption{First thirteen parameters of $p\ge 5$ and $m\ge 2$ for the matrix $A_{p,m}$ in (\ref{eq:Apm}), showing which result guarantees the invertibility of $A_{p,m}$. 
For each $p\ge 5$, any $A_{p,m}$ with  $m\geq \log(p-2)/\log((p-1)/(p-2))$ is nonsingular by Theorem~\ref{thm:Apm} from \cite{GMD}, and these parameters are shown with $\diamond$. 
For each $p$ of the form $p=2q+1$ with $q$ an odd prime, any $A_{p,m}$ is nonsingular by Theorem~\ref{thm:2q}, and these parameters are shown with $\star$.
For each $p$ of the form $p=4q+1$ with $q$ an odd prime, any $A_{p,m}$ with $2^q (\mod\; p)=0$ or $1$ in modulo $4$ is nonsingular by Theorem~\ref{thm:4q}, and these parameters are shown with $\star\star$.}
\label{table:pm}
\end{table}

From Theorem~\ref{thm:4q}, we see that, for example, when $p=13, 29$, $A_{p,m}$ is nonsingular for all odd $m\ge 3$. For the list of pairs $(q, p)$ up to $q \le 200$, the following seven cases $(3, 13)$, $(7, 29)$, $(37, 149)$, $ (43, 173)$, $(207, 509)$, $(163, 653)$, $(193, 773)$ satisfy the condition of the theorem.  There are six other cases up to $q \le 200$ with $(2^q (\mod\; p)) (\mod\; 4) = 3$. We cannot apply our theorem for these cases.

Table~\ref{table:pm} indicates the parameters $p$ and $m$ for which $A_{p,m}$ is invertible using different symbols for different results that guarantee the invertibility. We see that, although our results (Theorem~\ref{thm:2q}, Theorem~\ref{thm:4q}) in this section do not answer the question of invertibility of $A_{p,m}$ completely, they provide a good complement to the earlier result (Theorem~\ref{thm:Apm}).

\subsection{Application to circulant $0, 1$-matrices}
\label{subS:01circulant}
Various properties of circulant $0, 1$-matrices are particularly useful in Communication and Coding theory, Graph theory, and many other areas.  
In \cite{chen21} the author studied the condition for circulant $0, 1$-matrices to be nonsingular, which is useful in Communication theory and Coding theory. The main problem considered there is to determine the condition for the circulant $0, 1$-matrix with $k$ ones and $k+1$ zeros in each row to be nonsingular. Using associated polynomial and some properties of cyclotomic polynomials, the following result was obtained. If  $2k+1 = p^t$ with $p$ prime and $t$ positive integer, or $2k+1 = pq$ with $p$ and $q$ distinct odd primes, then such circulant matrix is always nonsingular regardless of arrangement of $k$ ones and $k+1$ zeros. On the other hand, when $2k+1=pqr$, with three distinct odd primes $p, q$ and $r$, there exists a circulant $0, 1$-matrix with $k$ ones and $k+1$ zeros that is singular.
\medskip 

For circulant $0, 1$-matrices, the number of ones (and so zeros as well) and the arrangements of ones and zeros determine invertibility. We apply our results to figure out invertibility of circulant $0, 1$-matrices with variations regarding on the size and arrangement. 

\begin{theorem} 
 Let $V$ be a circulant $0, 1$-matrix with $m$ ones and $n-m$ zeros in the first row. Suppose that $n = p^t$ for prime $p$ and positive integer $t$. If either $1 \le m \le p-1$ or $1 \le p^t-m \le p-1$, then the matrix $V$ is nonsingular.
\end{theorem} 

\begin{proof}
This is \ref{type:qk} or \ref{type:2k} in Section~\ref{sec:Explicit}. First of all, note that $\sum_{j = 1}^n v_j = m$ which is nonzero. Now suppose that $d = p^s$ for some $1 \le s \le t$. We show that the following identity cannot hold:

\begin{equation}
\label{eq:01Cond}
(p-1)\sum_{j = 0}^{p^{t-s}-1}v_{jp^s} 
= \sum_{j = 0}^{p^{t-s}-1}\left( v_{jp^s+p^{s-1}}+v_{jp^s+2p^{s-1}}+ \cdots +  v_{jp^s+(p-1)p^{s-1}}\right).
\end{equation}

Assume that $1 \le m \le p-1$. Without loss of generality we assume that $v_0 = 1$. If the identity (\ref{eq:01Cond}) holds, then the left-hand side is $(p-1)a$ and the right-hand side is at most $m-a$, where $a$ is the number of ones in the summation in the left-hand side, which implies $(p-1)a \le m-a$. Hence $p \le pa \le m \le p-1$ and this is a contradiction.  Therefore, the identity (\ref{eq:01Cond}) cannot hold, hence the circulant $0, 1$-matrix $V$ is nonsingular in this case. 

If $1 \le p^t-m \le p-1$, then there is at least one zero in each row so we assume that $v_0 = 0$. Since $v_0 = 0$ and the number of zeros is at most $p-1$, the number of zeros in the right hand side is at most $p-2$. It means that the right-hand side is at least $p^{t-s+1}-p^{t-s} - (p-2)$. Also the left-hand side is at most $(p-1)(p^{t-s} -1)$. If the identity (\ref{eq:01Cond}) holds, then $p^{t-s+1}-p^{t-s} - (p-2) \le (p-1)(p^{t-s} -1) $. But $p^{t-s+1}-p^{t-s}-p+ 2 > p^{t-s+1}-p^{t-s}-p+1$, hence we obtain a contradiction.
\end{proof}


%
%
%
%

\section*{Funding}
The first author is supported from the Institute for Basic Science (IBS-R029-C1). The second author acknowledges financial support from the National Research Foundation of Korea [Grant Numbers 2015R1A5A1009350 and 2021R1A2C1007598]. 

\section*{Notes}
Part of the work was performed while the second author was visiting the Korea Institute for Advanced Study, Seoul, Korea.

\bibliographystyle{tfnlm}
\bibliography{bib_cir}

\end{document}

%% file: macros.tex
\def\be{\begin{equation}}
\def\ee{\end{equation}}
\def\beaN{\setlength{\arraycolsep}{0.0em}\begin{eqnarray*}}
\def\eeaN{\end{eqnarray*}\setlength{\arraycolsep}{5pt}}
\def\bea{\setlength{\arraycolsep}{0.0em}\begin{eqnarray}}
\def\eea{\end{eqnarray}\setlength{\arraycolsep}{5pt}}

\def\mod{{\rm mod}\,}

\def\bks{\backslash}







%% file: main.bbl
\begin{thebibliography}{10}
\providecommand{\url}[1]{\normalfont{#1}}
\providecommand{\urlprefix}{Available from: }

\bibitem{GMD}
Hur~Y, Lubberts~Z. Generalizations of the {M}aillet determinant; 2014. Preprint
  (available in arXiv:1407.6802).

\bibitem{ZachThesis}
Lubberts~Z. Generating tight wavelet frames from sums of squares
  representations [dissertation]. Johns Hopkins University; 2019.

\bibitem{TM}
Muir~ST. Contributions to the history of determinants, 1900-1920. London:
  Blackie and Son Ltd.; 1930.

\bibitem{HWT}
Turnbull~HW. On certain modular determinants. Edinburgh Mathematical Notes.
  1940;\hspace{0pt}32.

\bibitem{CO}
Carlitz~L, Olson~FR. {M}aillet's determinant. Proc of the AMS.
  1955;\hspace{0pt}6(2).

\bibitem{CO2}
Carlitz~L. A generalization of {M}aillet's determinant and a bound for the
  first factor of the class number. Proc of the AMS. 1961;\hspace{0pt}12(2).

\bibitem{PCS}
Hur~Y, Zheng~F. Prime coset sum: A systematic method for designing multi-{D}
  wavelet filter banks with fast algorithms. IEEE Transactions on Information
  Theory. 2016;\hspace{0pt}62(11):6580--6593.

\bibitem{kra12}
Kra~I, Simanca~SR. On circulant matrices. Notices of the AMS.
  2012;\hspace{0pt}59(3):368--377.

\bibitem{HJ}
Horn~R, Johnson~C. Matrix analysis. 2nd ed. New York, NY: Cambridge University
  Press; 2013.

\bibitem{so06}
So~W. Integral circulant graphs. Discrete Mathematics.
  2006;\hspace{0pt}306(1):153--158.

\bibitem{chen21}
Chen~Z. On nonsingularity of circulant matrices. Linear Algebra and its
  Applications. 2021;\hspace{0pt}612:162--176.

\bibitem{lang}
Lang~S. Algebra. Springer New York; 2002.

\bibitem{gallian}
Gallian~J. Contemporary abstract algebra. Nelson Education; 2012.

\bibitem{nicol62}
Nicol~C. Some formulas involving {R}amanujan sums. Canadian Journal of
  Mathematics. 1962;\hspace{0pt}14:284--286.

\bibitem{lucht10}
Lucht~LG. A survey of {R}amanujan expansions. International Journal of Number
  Theory. 2010;\hspace{0pt}6(08):1785--1799.

\bibitem{anderson53}
Anderson~DR, Apostol~TM. The evaluation of {R}amanujan's sum and
  generalizations. Duke Mathematical Journal. 1953;\hspace{0pt}20(2):211--216.

\bibitem{hardy27}
Hardy~G, Seshu~Aiyan~P, Wilson~B. Collected papers of {S}rinivasa {R}amanujan.
  Cambridge, London. 1927;\hspace{0pt}.

\bibitem{park00}
Park~H, Park~J, Kim~D. A criterion on primitive roots modulo p. Journal of the
  Korean Society for Industrial and Applied Mathematics.
  2000;\hspace{0pt}4(1):29--38.

\bibitem{macmillan12}
MacMillan~K, Sondow~J. Divisibility of power sums and the generalized
  erd{\H{o}}s--moser equation. Elemente der Mathematik.
  2012;\hspace{0pt}67(4):182--186.

\end{thebibliography}
